\documentclass[12pt,color]{article}
\newtheorem{theorem}{Theorem}
\newtheorem{lemma}{Lemma}
\newtheorem{proposition}{Proposition}
\newtheorem{example}{Example}

\newtheorem{corollary}{Corollary}

\newtheorem{question}{Question}
\newtheorem{algorithm}{Algorithm}
\input amssym.tex
\usepackage{color}
\newtheorem{preproof}{{\bf Proof.}}

\newenvironment{proof}[1]{\begin{preproof}{\rm #1}\hfill{\rule[-0.5mm]
                         {2mm}{2mm}}}{\end{preproof}}
\def\n#1{\vbox to 3mm{\vspace{1mm}\vfill \hbox to 2.0mm{\hfill
             $#1$\hfill} \vfill }}

\def\mukm#1#2#3{(#1,#2,#3)\mbox{  Latin trade}}
\def\sfmukm#1#2#3{(#1,#2,#3)\mbox{ \sf Latin trade}}
\def\m#1#2{\raise 0.2ex\hbox{
    ${#1_{\displaystyle #2}}$}}
\def\x#1{\raise 0.5ex\hbox{
    ${#1}$}}
\def\arraystretch{1.0}                 

\def\m#1#2#3{\raise .8ex\hbox{
     ${#1_{{\bf \displaystyle #2}_{\displaystyle #3}}}$}}

\usepackage{array}
\title{{\bf On the existence of $3$--way $k$--homogeneous   Latin trades}}
\author{ Behrooz Bagheri Gh.${}^{a,b}$,
Diane Donovan${}^c$, and E. S. Mahmoodian$^{b}$}
\date{}

\begin{document}
\maketitle
\begin{center}
$a$
Department of Mathematical Sciences\\
Isfahan University of Technology\\
84156-83111,
\vspace*{5mm}
 Isfahan, I. R. Iran  \\
$b$
Department of Mathematical Sciences \\
Sharif University of Technology \\
P. O. Box 11155--9415,
\vspace*{5mm}
Tehran, I. R. Iran \\
$c$
Department of Mathematics\\
The University of Queensland\\
Brisbane 4072\\
Australia \\
\end{center}
%
\begin{abstract}
A {\sf $\mu$-way Latin trade} of volume $s$ is a collection of $\mu$ partial Latin squares
$T_1,T_2,\ldots,T_{\mu}$, containing exactly the same $s$ filled
cells, such that if cell $(i, j)$ is filled, it contains a
different entry in each of the $\mu$ partial Latin squares, and
such that row $i$ in each of the $\mu$ partial Latin squares
contains, set-wise, the same symbols and column $j$, likewise.
It is called
{\sf $\mu$--way $k$--homogeneous Latin trade}, if in each row and
each column $T_r$, for $1\le r\le \mu,$ contains exactly $k$ elements, and each element
appears in $T_r$ exactly $k$ times. It is also denoted by $(\mu,k,m)$ Latin
 trade, where $m$ is the size of partial Latin squares.

We introduce some general constructions for $\mu$--way
$k$--homogeneous Latin trades and specifically show that for all
$k \le m$, $6\le k \le 13$ and $k=15$, and for all $k \le m$, $k =
4, \ 5$ (except for four specific values), a $3$--way
$k$--homogeneous Latin trade of volume $km$ exists. We also show
that there are no $(3,4,6)$ Latin trade and $(3,4,7)$ Latin trade. Finally we
present general results on the existence of $3$--way
$k$--homogeneous Latin trades for some modulo classes of $m$.
\end{abstract}

\noindent {\bf AMS Subject Classification:} 05B15 \\ {\bf
Keywords:} Latin square;  Latin trade; $\mu$--way Latin trade;
$\mu$--way $k$--homogeneous Latin trade.

\section{Introduction}     
A {\sf Latin square} $L$ of order $n$ is an $n\times n$ array usually on
the set $N=\{1,\ldots,n\}$ where each element of $N$ appears
exactly once in each row and exactly once in each column. We can
represent each Latin square as a subset of $N\times N\times N$,
$$L=\{(i,j;k)\mid \mbox{ element $k$ is located in position }(i,j)\}.$$
\noindent A {\sf partial Latin square} $P$ of order $n$ is an
$n\times n$ array of elements from the set $N$, where each
element of $N$ appears at most once in each row and at most once
in each column. The set $S_P=\{(i,j)\mid (i,j;k)\in P\}$ of the
partial Latin square $P$ is called the {\sf shape } of $P$ and
$|S_P|$ is called the {\sf volume} of $P$. By ${\cal R}_P^i$ and
${\cal C}_P^j$ we mean the set of entries in row $i$ and column
$j$, respectively of $P$. A {\sf $\mu$-way Latin trade},
$(T_1,T_2,\ldots,T_{\mu})$, of volume $s$ is a collection of $\mu$ partial Latin squares
$T_1,T_2,\ldots,T_{\mu}$, containing exactly the same $s$ filled
cells, such that if cell $(i, j)$ is filled, it contains a
different entry in each of the $\mu$ partial Latin squares, and
such that row $i$ in each of the $\mu$ partial Latin squares
contains, set-wise, the same symbols and column $j$, likewise. If
$\mu=2$, $(T_1,T_2)$ is called a {\sf Latin bitrade}. The study
of Latin trades and combinatorial trades in general, has
generated much interest in recent years. For a survey on the
topic see \cite{MR2041871},
 \cite{MR2048415},  and \cite{MR2453264}.

A $\mu$--way Latin trade which is obtained from another one by
deleting its empty rows and empty columns, is called a {\sf
$\mu$--way $k$--homogeneous Latin trade} $(\mu \le k)$ or briefly
a $\sfmukm{\mu}{k}{m}$, if it has $m$ rows and in each row and
each column $T_r$, for $1\le r\le \mu,$ contains exactly $k$ elements, and each element
appears in $T_r$ exactly $k$ times.

In Figure~1($a$) a $\mukm{3}{5}{7}$ is demonstrated. The elements
of $T_2$ and $T_3$ are written as subscripts in the same array as
$T_1$.  ($\bullet$ means the cell is empty.)
%
\def\arraystretch{1.4}
\begin{center}
\begin{tabular}
{|@{\hspace{1pt}}c@{\hspace{1pt}} |@{\hspace{1pt}}c@{\hspace{1pt}}
|@{\hspace{1pt}}c@{\hspace{1pt}} |@{\hspace{1pt}}c@{\hspace{1pt}}
|@{\hspace{1pt}}c@{\hspace{1pt}} |@{\hspace{1pt}}c@{\hspace{1pt}}
|@{\hspace{1pt}}c@{\hspace{1pt}}|} \hline
\m{1}{2}{3}&\m{3}{5}{2}&\m{5}{3}{7}&\m{7}{1}{5}&\m{2}{7}{1}&\m{}{}{\bullet}&\m{}{}{\bullet}\\\hline
\m{}{}{\bullet}&\m{2}{3}{4}&\m{4}{6}{3}&\m{6}{4}{1}&\m{1}{2}{6}&\m{3}{1}{2}&\m{}{}{\bullet}\\\hline
\m{}{}{\bullet}&\m{}{}{\bullet}&\m{3}{4}{5}&\m{5}{7}{4}&\m{7}{5}{2}&\m{2}{3}{7}&\m{4}{2}{3}\\\hline
\m{5}{3}{4}&\m{}{}{\bullet}&\m{}{}{\bullet}&\m{4}{5}{6}&\m{6}{1}{5}&\m{1}{6}{3}&\m{3}{4}{1}\\\hline
\m{4}{5}{2}&\m{6}{4}{5}&\m{}{}{\bullet}&\m{}{}{\bullet}&\m{5}{6}{7}&\m{7}{2}{6}&\m{2}{7}{4}\\\hline
\m{3}{1}{5}&\m{5}{6}{3}&\m{7}{5}{6}&\m{}{}{\bullet}&\m{}{}{\bullet}&\m{6}{7}{1}&\m{1}{3}{7}\\\hline
\m{2}{4}{1}&\m{4}{2}{6}&\m{6}{7}{4}&\m{1}{6}{7}&\m{}{}{\bullet}&\m{}{}{\bullet}&\m{7}{1}{2}\\\hline
\end{tabular}
\hspace*{10mm}
\begin{tabular}
{|@{\hspace{1pt}}c@{\hspace{1pt}} |@{\hspace{1pt}}c@{\hspace{1pt}}
|@{\hspace{1pt}}c@{\hspace{1pt}} |@{\hspace{1pt}}c@{\hspace{1pt}}
|@{\hspace{1pt}}c@{\hspace{1pt}}
|@{\hspace{4.5pt}}c@{\hspace{4.5pt}}
|@{\hspace{7.1pt}}c@{\hspace{7.1pt}}|} \hline
\m{1}{2}{3}&\m{3}{5}{2}&\m{5}{3}{7}&\m{7}{1}{5}&\m{2}{7}{1}&\m{}{}{\bullet}&\m{}{}{\bullet}\\\hline
\m{}{}{\bullet}&$\searrow$&$\searrow$&$\searrow$&$\searrow$&$\searrow$&\m{}{}{\bullet}\\\hline
\m{}{}{\bullet}&\m{}{}{\bullet}&\m{}{}{}&\m{}{}{}&\m{}{}{}&\m{}{}{}&\m{}{}{}\\\hline
\m{}{}{}&\m{}{}{\bullet}&\m{}{}{\bullet}&\m{}{}{}&\m{}{}{}&\m{}{}{}&\m{}{}{}\\\hline
\m{}{}{}&\m{}{}{}&\m{}{}{\bullet}&\m{}{}{\bullet}&\m{}{}{}&\m{}{}{}&\m{}{}{}\\\hline
\m{}{}{}&\m{}{}{}&\m{}{}{}&\m{}{}{\bullet}&\m{}{}{\bullet}&\m{}{}{}&\m{}{}{}\\\hline
\m{}{}{}&\m{}{}{}&\m{}{}{}&\m{}{}{}&\m{}{}{\bullet}&\m{}{}{\bullet}&\m{}{}{}\\\hline
\end{tabular}
\\\vspace*{1.9mm}
\hspace*{-1.0mm} ($a$) \hspace*{59mm} ($b$) \\
\end{center}
\begin{center}
\begin{figure}[ht]
\label{circulant} \vspace*{-6mm} \caption{A $\mukm{3}{5}{7}$ and
its base row}
\end{figure}
\end{center}
\vspace*{-7mm}
A $\mukm{\mu}{k}{m}$  $(T_1,T_2,\ldots,T_{\mu})$ is called {\sf
circulant}, if it can be obtained from the elements of its first
row, called the {\sf base row} and denoted by $\mu$--$B_{m}^{k}$,
 by permuting the coordinates cyclically along the diagonals. For example in Figure~1($b$), a  $3$--$B_{7}^{5}$ base row, $\{(1,2,3)_1,(3,5,2)_2,(5,3,7)_3,(7,1,5)_4,(2,7,1)_5\}$, is shown.
Actually if a base row $B=\{(a_1,a_2,\ldots,a_{\mu})_{c_l}\mid 1\le l\le k \}$, where $a_r$ and $c_l$ $\in \{1,2,\ldots, m \},$ is given, we construct a set of $\mu$ partial Latin squares as in the following manner:
$$ 1\le r \le \mu, \ T_r=\{ (1+i,c_l+i;a_r+i)({\rm mod} \ m)| 0\le i \le m-1, 1\le l \le k \}.$$ 
\begin{algorithm}
\label{base-algorithm}
To check that $B=\{(a_1,a_2,\ldots,a_{\mu})_{c_l}\mid 1\le l\le k \}$, where $a_r$ and $c_l$ $\in \{1,2,\ldots, m \},$ is a base row of a $\mukm{\mu}{k}{m}$:\\
we note that for each $r$, $ 1\le r \le \mu$, ${\cal R}_{T_r}^1=\{a_r\mid  (a_1,a_2,\ldots,a_{\mu})_{c_l}\in B \ {\rm and} \ 1\le l\le k\} $ and                                                            
${\cal C}_{T_r}^m=\{a_r+m-c_l\ \equiv a_r-c_l ({\rm mod} \ m)\mid  (a_1,a_2,\ldots,a_{\mu})_{c_l}\in B  \ {\rm and} \ 1\le l\le k \}$.                                             
Now if $B$ satisfies  the following conditions, then it will suffice to be a base row of a $\mukm{\mu}{k}{m}$.                   
\begin{description}                                                                                                                     
\item{\rm{(i)}} $a_r$'s are distinct, for each $(a_1,a_2,\ldots,a_{\mu})_{c_l}\in B$.                                                       
\item{\rm{(ii)}} $c_l$'s are distinct.                                                                                                       
\item{\rm{(iii)}} ${\cal R}_{T_1}^1={\cal R}_{T_2}^1=\cdots={\cal R}_{T_\mu}^1$.                                                             
\item{\rm{(iv)}} ${\cal C}_{T_1}^m={\cal C}_{T_2}^m=\cdots={\cal C}_{T_\mu}^m$.                                                              
\end{description}
\end{algorithm}                                                                                                              
\begin{lemma}\label{k,k}
For each $k\ge \mu$, a $\mukm{\mu}{k}{k}$ exists.
\end{lemma}
\begin{proof}{
By taking a Latin square of order $k$ and permuting its
rows, cyclically, $\mu$ times we obtain the desired Latin trade.
}\end{proof}

A $(\mu,\mu,\mu)$ Latin trade  is called a {\sf
$\mu$--intercalate}.

\def\arraystretch{1.4}
\begin{center}
\begin{tabular}
{|@{\hspace{1pt}}c@{\hspace{1pt}} |@{\hspace{1pt}}c@{\hspace{1pt}}
|@{\hspace{1pt}}c@{\hspace{1pt}}|} \hline
\m{1}{3}{2}&\m{2}{1}{3}&\m{3}{2}{1}\\\hline
\m{3}{2}{1}&\m{1}{3}{2}&\m{2}{1}{3}\\\hline
\m{2}{1}{3}&\m{3}{2}{1}&\m{1}{3}{2}\\\hline
\end{tabular}
\end{center}
\begin{center}
\begin{figure}[ht]
\label{3x3} \vspace*{-7mm} \caption{\label{m_s}A $3$--intercalate}
\end{figure}
\end{center}
The following question is of interest.
\begin{question}
\label{existence} For given $m$ and $k$,  $m \ge k\ge \mu$, does
there exist a $\mukm{\mu}{k}{m}$?
\end{question}
For Latin bitrades, Question~\ref{existence} is discussed and is
answered completely in ~\cite{MR2170114},~\cite{MR2139816},
~\cite{MR2220235},~\cite{BagheriMah}, and~\cite{MR2563279}. In
this paper applying earlier results we introduce some general
constructions for $\mukm{\mu}{k}{m}$s and specifically
concentrate on the case of $\mu=3$.  Our main result is stated in the following theorem.
%
\begin{theorem}\label{main}
 All $\mukm{3}{k}{m}$s $(m\ge k\ge 3)$ exist, for
\begin{itemize}
\item $k=4$, except  for $m=6$ and $7$ and possibly for $m=11$,
\item $k=5$, except possibly for $m=6$, 
\item
$6\le k \le 13$,
\item
$k=15$,  
\item
$k\ge 4$ and $m\ge k^2$,
\item
$m$ a multiple of $5$, except possibly for $m=30$,
\item
$m$ a multiple of $7$, except possibly for $m=42$ and
$\mukm{3}{4}{7}$.

\end{itemize}
\end{theorem}

\section{General constructions}         
\begin{theorem}\label{sum}
If $\ $ $l \neq 2, 6$ and for each $k\in\{k_1,\ldots,k_l\}$ there
exists a $(\mu,k,p)$ Latin trade, then a
$\mukm{\mu}{k_1+\cdots+k_l}{lp}$ exists. {\rm(}Some $k_i$s can
possibly be zero.{\rm)}
\end{theorem}

\begin{proof}{
Since  $l \neq 2, 6$, there exist two $l\times l$ orthogonal Latin
squares. Denote these Latin squares
 by $L_1$ and $L_2$, with elements chosen from the sets $\{e_1,e_2,\ldots,e_l\}$ and
$\{f_1,f_2,\ldots,f_l\}$, respectively. Assume that $L^*$ is a
square that is formed by superposing $L_1$ and $L_2$.  We replace
each $(e_i,f_j)$ in $L^*$ with a $\mukm{\mu}{k_j}{p}$   whose
elements are from the set $\{(i-1)p+1,(i-1)p+2,\ldots,ip\}$. As a
result we obtain a $\mukm{\mu}{k_1+\cdots+k_l}{lp}$. }\end{proof}

\begin{theorem}\label{k(k+1)}
If the number of mutually orthogonal Latin squares of order $k+1$,
${\rm MOLS}(k+1)$, is greater than or equal to $\mu+1$, then there
exists a $(\mu,k,k+1)$ Latin trade.
\end{theorem}
\begin{proof}{
By Exercise $5.2.11$ of~\cite{Lindner}~page $103$, there are $\mu$
idempotent ${\rm MOLS}(k+1).$ If in each of those ${\rm MOLS}$ we delete the
main diagonals, we obtain a $(\mu,k,k+1)$ Latin trade.}~\end{proof}


Actually by applying results of existence of idempotent
${\rm MOLS}(n)$~(\cite{MR2246267}, Section~$3.6$, Table~$3.83$), we can
improve Theorem~\ref{k(k+1)} for the case $\mu=3$ as  follows.

\begin{theorem}\label{k+1}
If $k\ge 11$, then there exists a $\mukm{3}{k}{k+1}$.
\end{theorem}
\begin{theorem}\label{intercalate}
Any $\mukm{\mu}{\mu}{m}$, ${\bf T}=(T_1,T_2,\ldots,T_{\mu})$, can
be partitioned into disjoint $\mu$--intercalates.
\end{theorem}

\begin{proof}{
We prove this result by induction.
Without loss of generality, let $(1,1;r)\in T_r$ for each $1\le r\le \mu$.
Therefore $\{1,2,\ldots,\mu\}\subset {\cal R}_{T_r}^i\cap {\cal C}_{T_r}^i$ for each $1\le i,r\le \mu$. Since $|{\cal R}_{T_r}^i|=|{\cal C}_{T_r}^i|=\mu$ for each $1\le i,r\le \mu$,
 ${\cal R}_{T_r}^i={\cal C}_{T_r}^i=\{1,2,\ldots,\mu\}$ for each $1\leq i,r \leq \mu$.
 Again without loss of generality, let $(i,1;i)\in T_1$ and
$(1,j;j)\in T_1$ for $1\leq i,j\leq \mu$. This implies that
$\{(i,j)\mid 1\leq i,j\leq\mu\}$ is a subset of shape of $T_1$.
Therefore subarray $\{(i,j)\mid 1\leq i,j\leq\mu\}$ with elements
$\{1,2,\ldots,\mu\}$ is a $\mu$--intercalate. We can apply the same
 argument to the $(m-\mu)\times (m-\mu)$ subsquare obtained by removing rows
 $1,2,\ldots,\mu$  and
columns $1,2,\ldots,\mu$. This completes the proof. }\end{proof}
\begin{corollary}\label{khom}
For every $m\geq 1$, there exists a $\mukm{\mu}{k}{m}$ with $k=\mu$,
if and only if $k|m$.
\end{corollary}
%
%
\begin{theorem}\label{m_1m_2}
Assume that $m_i\geq k_i$, for $i=1,2$. If there exists a
$(\mu_i,k_i,m_i)$ Latin trade for $i=1,2$, then there exists a
$\mukm{\mu_1\mu_2}{k_1k_2}{m_1m_2}$.
\end{theorem}
\begin{proof}{
We construct a $\mukm{\mu_1\mu_2}{k_1k_2}{m_1m_2}$
in the following way:\\
Suppose $(T_1,T_2,\ldots,T_{\mu_1})$ is a $\mukm{\mu_1}{k_1}{m_1}$
 and ${\bf U}=(U_1,U_2,\ldots,U_{\mu_2})$ is a\\ $\mukm{\mu_2}{k_2}{m_2}$. For each entry $i$ in
$T_1,T_2,\ldots,T_{\mu_1}$, we replace $i$ with a copy of
 ${\bf U}$
 where elements are chosen
from the set $\{(i-1)m_2+1,(i-1)m_2+2,\ldots,im_2\}$; replace the
empty cells in $T_1,T_2,\ldots,T_{\mu_1}$  with an empty
$m_2\times m_2$ array. As a result  we obtain a
$\mukm{\mu_1\mu_2}{k_1k_2}{m_1m_2}$.
}\end{proof}
\begin{corollary}\label{4-way}
Suppose $k=k_1k_2$ and $m=m_1m_2$ where $m_i \ge k_i \ge 2$, for
$i=1,2$.  Then there exists a $\mukm{4}{k}{m}$, provided that  if $k_j=2$,
for some $j$, then $m_j$ must be assumed to be even.
\end{corollary}
\begin{proof}{
It is shown that Latin homogeneous bitrades (i.e
$\mukm{2}{k}{m}$) exist for all $m\ge k\ge 3$ and for all even
$m$, when $k=2$.
(See~\cite{MR2170114},~\cite{MR2139816},~\cite{MR2220235},~\cite{BagheriMah},
and~\cite{MR2563279}.) }\end{proof}
\begin{theorem}\label{k(rm+sn)}
For every $k$, if there exists a $\mukm{\mu}{k}{m}$ and a
$(\mu,k,n)$ Latin trade, then
there exists a $\mukm{\mu}{k}{m+n}$.
\end{theorem}
\begin{proof}{
Let ${\bf T_1}$ 
be a $\mukm{\mu}{k}{m}$ and ${\bf T_2}$
 be a $\mukm{\mu}{k}{n}$ such
that the elements of ${\bf T_1}$ are in
the set $\{1, \ldots, m\}$ and
the elements of ${\bf T_2}$ are chosen from the set
$\{m+1, \ldots, m+n\}$.  Therefore, the following Latin trade is a
$\mukm{\mu}{k}{m+n}$. \vspace*{-0.7cm}
\def\arraystretch{1.25}
\begin{center}
$$\begin{array}
{|@{\hspace{2pt}}c@{\hspace{2pt}}@{\hspace{1pt}}c@{\hspace{1pt}}
 @{\hspace{1pt}}c@{\hspace{1pt}}@{\hspace{1pt}}c@{\hspace{1pt}}
|} \hline 
{\bf T_1}&&&\\
&{\bf T_2}&&\\
 \hline
\end{array}$$
\end{center} \vspace*{-0.8cm}  }\end{proof}

\begin{corollary}\label{k^2}
If  the number of ${\rm MOLS}(k+1)\geq \mu +1$,  then for each $m$ where $m\geq k^2$,
there exists a $\mukm{\mu}{k}{m}$.
\end{corollary}

\begin{proof}{
If $m\geq k^2$, then we can write $m$ as $m= rk+s(k+1)$, where
$r,s\geq 0$.  Theorem~\ref{k(rm+sn)} and Theorem~\ref{k(k+1)} lead
us to a conclusion. }\end{proof}
%

By Theorems~\ref{k(rm+sn)} and~\ref{k+1} we have:
\begin{corollary}
If $k\geq 11$, then for each $m$ where $m\geq k^2$, there exists
a $(3,k,m)$ Latin trade.
\end{corollary}

\begin{theorem}\label{2k-1}
Consider an arbitrary natural number $k$. If  for every ${k+1}\leq
l\leq 2k-1$ there exists a $\mukm{\mu}{k}{l}$, then for any $m
\geq k$ there exists a $(\mu,k,m)$ Latin trade.
\end{theorem}
\begin{proof}{
For every $m \geq 2k$, we can write $m=rk+sl$, where $r,s \geq 0$
and ${k+1}\leq l\leq 2k-1$. Since there exist a $\mukm{\mu}{k}{k}$
and a $\mukm{\mu}{k}{l}$, by Theorem~\ref{k(rm+sn)} we conclude
that there exists a $(\mu,k,m)$ Latin trade. }\end{proof}
%
\section{$\mu=3$ }     
 In this section we apply the above constructions to establish the
 existence of $3$--way
$k$--homogeneous Latin trades for specific values of $k$,
 and when $m$ is a multiple of $5$ or
 $7$. We also show that there is no $(3,4,6)$ Latin
 trade.
%
\subsection{Small even $k$ }
\begin{proposition}
There exists a $\mukm{3}{4}{m}$ for every $m\ge 4$, except possibly
for $m=6,7$ and $11$.
\end{proposition}
\begin{proof}{
By Lemma~\ref{k,k} and Theorem~\ref{k(k+1)} there exist a $\mukm{3}{4}{4}$ and a $(3,4,5)$ Latin trade, respectively.
Since \ $8=2\times 4, \ 9=4+5, \ 10=2\times 5, \ 12=3\times 4, \ 13=2\times 4+5, \
14=4+2\times 5,$ \ and \ $15=3\times 5$;
Theorem~\ref{k(rm+sn)} results that there exist $\mukm{3}{4}{m}$s for
$m=8,9,10,12,13,14$, and $15$. Since the number $ {\rm MOLS}(5)=4$, then by
Corollary~\ref{k^2} there exists a $(3,4,m)$ Latin trade, for every
$m\ge 16$. }\end{proof}
\begin{proposition}
\label{(3,4,6)}
There is no $(3,4,6)$ Latin trade.
\end{proposition}
\begin{proof}{By contradiction. Suppose $T=(T_1,T_2,T_3)$ is a $\mukm{3}{4}{6}$.
By applying some permutations on rows and columns,
 if necessary, we may assume that all cells containing the element
 1 form a $4\times 4$ array  minus  a transversal $\tau$, which will be labeled $L$. For example in Figure~3 one of the possible positions of 1 is shown. Note that there are 12 cells in $L$ each of which has a 1 in one of the $T_i$'s. In what follows the argument is based only on the assumption that in each of those cells there exists
 one 1 from one of the $T_i$'s. ($\bullet$ means the cell is empty.)
%
\vspace*{-4mm}
\def\arraystretch{1}
\begin{center}
\begin{tabular}
{|@{\hspace{1pt}}c@{\hspace{1pt}} |@{\hspace{1pt}}c@{\hspace{1pt}}
|@{\hspace{1pt}}c@{\hspace{1pt}} |@{\hspace{1pt}}c@{\hspace{1pt}}
||@{\hspace{8pt}}c@{\hspace{8pt}}
|@{\hspace{8pt}}c@{\hspace{8pt}}|} \hline
\m{1}{.}{.}&\m{.}{1}{.}&\m{.}{.}{1}&\m{}{}{\bullet}
&\m{}{}{}&\m{}{}{}\\\hline \m{.}{1}{.}&\m{.}{.}{1}&\m{}{}{\bullet}
&\m{1}{.}{.}&\m{}{}{}&\m{}{}{}\\\hline
\m{.}{.}{1}&\m{}{}{\bullet}&\m{1}{.}{.}&\m{.}{1}{.}&\m{}{}{}&\m{}{}{}\\\hline
\m{}{}{\bullet}&\m{1}{.}{.}&\m{.}{1}{.}&\m{.}{.}{1}&\m{}{}{}&\m{}{}{}\\\hline\hline
\m{}{}{}&\m{}{}{}&\m{}{}{}&\m{}{}{}&\m{}{}{}&\m{}{}{}\\\hline
\m{}{}{}&\m{}{}{}&\m{}{}{}&\m{}{}{}&\m{}{}{}&\m{}{}{}\\\hline
\end{tabular}
\end{center}
\vspace*{-7mm}
\begin{center}
\begin{figure}[ht]
\label{346} 
\caption{Positions of 1 in
$T=(T_1,T_2,T_3)$}
\end{figure}
\end{center}
\vspace*{-10mm}

In the first stage we show that the cells of $\tau$ in $T$ are
empty.
%
Suppose without loss of generality the cell $T_{14}$ in $\tau$ is
not empty. Then $T_{54}$ and $T_{64}$ must be empty. Thus at least
4 cells of $\{T_{51}, T_{52}, T_{53},T_{61}, T_{62}, T_{63}\}$
must be filled. Then by pigeonhole principal there exists a column
in $T$ with at least 5 filled cells, a contradiction. So all cells
of $\tau$ are empty. Therefore exactly 4 cells of $\{T_{51},
T_{52}, T_{53}, T_{54},T_{61}, T_{62}, T_{63}, T_{64}\}$ are
filled, and from $T$ being $4$-homogeneous all the cells:
$\{T_{55}, T_{56},T_{65}, T_{66}\}$ are filled.

In the second stage we show that no element, other than 1, appears
more than two times in any row or in any column of $L$. For example
let us denote by $\{1,x,y,z\}$, the elements which appear in the
first row and without loss of generality $T_{15}$ is another filled
cell of that row. In contrary, assume that $x$ appears three times
in the first row of $L$, i.e. in the cells $T_{11}, T_{12},$ and
$T_{13}$. This leaves only two elements $y$ and $z$ to appear in
$T_{15}$, which is a contradiction for $T$ being a $3$-way Latin
trade. So each of the elements other than 1, either does not appear
in a row of $L$  or it appears exactly two times in a row of $L$.
Now each element other than 1 if it appears in $L$, it occupies 4,
6, or 8 cells.

In the third stage we show that no element occupies 6 or  8 cells
of $L$. If an element, say $u \neq 1$ appears 8 times in $L$, then
since $u$ appears 2 times in each row and in each column of $L$,
so it appears once in each row of  the
$[1,\ldots,4]\times[5,6]$ block. This means that
$u$ appears at least 16 times in $T$, which is a contradiction. If
$u \neq 1$ appears 6 times in $L$ then three rows and three
columns of $L$ each contains $u$ twice. So without loss of
generality one of the following cases happens.

\begin{center}
\begin{tabular}
{|@{\hspace{1pt}}c@{\hspace{1pt}} |@{\hspace{1pt}}c@{\hspace{1pt}}
|@{\hspace{1pt}}c@{\hspace{1pt}} |@{\hspace{1pt}}c@{\hspace{1pt}}
||@{\hspace{5pt}}c@{\hspace{5pt}}
|@{\hspace{5pt}}c@{\hspace{5pt}}|} \hline
\m{1}{.}{.}&\m{.}{1}{.}&\m{.}{.}{1}&\m{}{}{\bullet}
&\m{}{u}{}&\m{}{}{\bullet}\\\hline \m{.}{1}{.}&\m{.}{.}{1}&\m{}{}{\bullet}
&\m{1}{.}{.}&\m{}{u}{}&\m{}{}{\bullet}\\\hline
\m{.}{.}{1}&\m{}{}{\bullet}&\m{1}{.}{.}&\m{.}{1}{.}&\m{}{u}{}&\m{}{}{\bullet}\\\hline
\m{}{}{\bullet}&\m{1}{.}{.}&\m{.}{1}{.}&\m{.}{.}{1}&\m{}{}{}&\m{}{}{}\\\hline\hline
\m{}{}{}&\m{}{}{}&\m{}{}{}&\m{}{}{}&\m{}{}{}&\m{}{}{}\\\hline
\m{}{}{}&\m{}{}{}&\m{}{}{}&\m{}{}{}&\m{}{}{}&\m{}{}{}\\\hline
\end{tabular}
\hspace*{3cm}
\begin{tabular}
{|@{\hspace{1pt}}c@{\hspace{1pt}} |@{\hspace{1pt}}c@{\hspace{1pt}}
|@{\hspace{1pt}}c@{\hspace{1pt}} |@{\hspace{1pt}}c@{\hspace{1pt}}
||@{\hspace{5pt}}c@{\hspace{5pt}}
|@{\hspace{5pt}}c@{\hspace{5pt}}|} \hline
\m{1}{.}{.}&\m{.}{1}{.}&\m{.}{.}{1}&\m{}{}{\bullet}
&\m{}{u}{}&\m{}{}{\bullet}\\\hline \m{.}{1}{.}&\m{.}{.}{1}&\m{}{}{\bullet}
&\m{1}{.}{.}&\m{}{u}{}&\m{}{}{\bullet}\\\hline
\m{.}{.}{1}&\m{}{}{\bullet}&\m{1}{.}{.}&\m{.}{1}{.}&\m{}{}{\bullet}&\m{}{u}{}\\\hline
\m{}{}{\bullet}&\m{1}{.}{.}&\m{.}{1}{.}&\m{.}{.}{1}&\m{}{}{}&\m{}{}{}\\\hline\hline
\m{}{}{}&\m{}{}{}&\m{}{}{}&\m{}{}{}&\m{}{}{}&\m{}{}{}\\\hline
\m{}{}{}&\m{}{}{}&\m{}{}{}&\m{}{}{}&\m{}{}{}&\m{}{}{}\\\hline
\end{tabular}
\\\vspace*{1.3mm}
\hspace*{-2mm} ($a$) \hspace*{60mm} ($b$) \\
\end{center}
\begin{center}
\begin{figure}[ht]
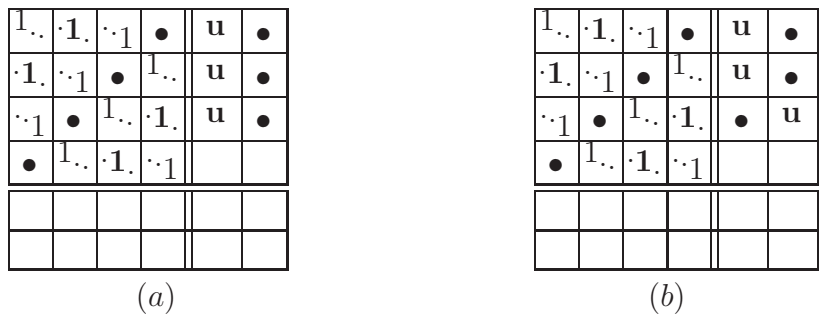

\label{circulant} \vspace*{-10mm} \caption{Positions of u in the
fifth and sixth columns of $T$}
\end{figure}
\end{center}
\vspace*{-10mm} In case ($a$) the fifth column has at least 5
filled cells
which is a contradiction. In case ($b$) there are  five columns of $T$
which have $u$ and since each column containing an $u$ will
contain 3 of them, so there are  at least 15 cells containing $u$
in $T$, which is a contradiction.

Now we have shown that each $u \neq 1$ if it appears in $L$, it
appears exactly 4 times. The array $L$ has exactly $36-12=24$
places for elements different from 1 to occupy while the 5 other elements
can fill at most $5\times 4=20$ places, which is a contradiction.
%
}\end{proof}
\begin{proposition}
\label{(3,4,7)}
There is no $(3,4,7)$ Latin trade.
\end{proposition}
\begin{proof}{By contradiction. Suppose $T=(T_1,T_2,T_3)$ is a $\mukm{3}{4}{7}$.
By applying some permutations on rows and columns,
 if necessary, we may assume that all cells containing the element
 1 form a $4\times 4$ array  minus  a transversal $\tau$, which will be labeled $L$. For example in Figure~5 one of the possible positions of 1 is shown. Note that there are 12 cells in $L$ each of which has a 1 in one of the $T_i$'s. In what follows the argument is based only on the assumption that in each of those cells there exists
 one 1 from one of the $T_i$'s. ($\bullet$ means the cell is empty.)
 \vspace*{-4mm}
\def\arraystretch{1}
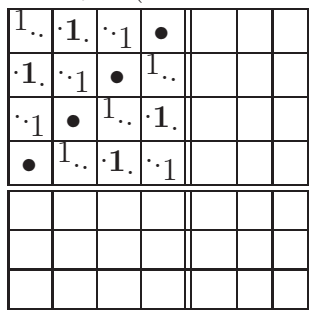
\begin{center}
\begin{tabular}
{|@{\hspace{1pt}}c@{\hspace{1pt}} |@{\hspace{1pt}}c@{\hspace{1pt}}
|@{\hspace{1pt}}c@{\hspace{1pt}} |@{\hspace{1pt}}c@{\hspace{1pt}}
||@{\hspace{8pt}}c@{\hspace{8pt}}
|@{\hspace{6pt}}c@{\hspace{6pt}}|@{\hspace{6pt}}c@{\hspace{6pt}}|}\hline
\m{1}{.}{.}&\m{.}{1}{.}&\m{.}{.}{1}&\m{}{}{\bullet}
&\m{}{}{}&\m{}{}{}&\m{}{}{}\\\hline \m{.}{1}{.}&\m{.}{.}{1}&\m{}{}{\bullet}
&\m{1}{.}{.}&\m{}{}{}&\m{}{}{}&\m{}{}{}\\\hline
\m{.}{.}{1}&\m{}{}{\bullet}&\m{1}{.}{.}&\m{.}{1}{.}&\m{}{}{}&\m{}{}{}&\m{}{}{}\\\hline
\m{}{}{\bullet}&\m{1}{.}{.}&\m{.}{1}{.}&\m{.}{.}{1}&\m{}{}{}&\m{}{}{}&\m{}{}{}\\\hline\hline
\m{}{}{}&\m{}{}{}&\m{}{}{}&\m{}{}{}&\m{}{}{}&\m{}{}{}&\m{}{}{}\\\hline
\m{}{}{}&\m{}{}{}&\m{}{}{}&\m{}{}{}&\m{}{}{}&\m{}{}{}&\m{}{}{}\\\hline
\m{}{}{}&\m{}{}{}&\m{}{}{}&\m{}{}{}&\m{}{}{}&\m{}{}{}&\m{}{}{}\\\hline
\end{tabular}
\end{center}
\vspace*{-7mm}
\begin{center}
\begin{figure}[ht]
\label{346} 
\caption{Positions of 1 in
$T=(T_1,T_2,T_3)$}
\end{figure}
\end{center}
\vspace*{-10mm}

If we focus on the placement of the remaining filled cells in $T$, we see that rows $1$ to $4$ of $T$ each have
one additional filled cell in one of columns $5,6$ or $7$. Likewise for columns $1$ to $4$ of rows $5,6$ or $7$.
Further, the subsquare defined by the intersection of rows $5,6,$ and $7$ with columns $5,6,$ and $7$, can have 
at most three filled cells in any row or column. Hence it follows that without loss of generality columns $5$ has two 
filled cell in rows $1$ to $4$ (similarly row $5$ has two filled cells in columns $1$ to $4$) and columns $6$ and 
$7$ have one filled cell in rows $1$ to $4$ (similarly rows $6$ and $7$ have one filled cell in columns $1$ to $4$).
Thus we may assume cell $(5,5)$ is empty and one possible distribution of empty cells (one out of $36$) is:

\def\arraystretch{1}
\begin{center}
\begin{tabular}
{|@{\hspace{5pt}}c@{\hspace{5pt}} |@{\hspace{5pt}}c@{\hspace{5pt}}
|@{\hspace{5pt}}c@{\hspace{5pt}} |@{\hspace{5pt}}c@{\hspace{5pt}}
||@{\hspace{5pt}}c@{\hspace{5pt}}
|@{\hspace{5pt}}c@{\hspace{5pt}}|@{\hspace{5pt}}c@{\hspace{5pt}}|}\hline
\m{}{}{}&\m{}{}{}&\m{}{}{}&\m{}{}{\bullet}
&\m{}{}{}&\m{}{}{\bullet}&\m{}{}{\bullet}\\\hline \m{}{}{}&\m{}{}{}&\m{}{}{\bullet}
&\m{}{}{}&\m{}{}{}&\m{}{}{\bullet}&\m{}{}{\bullet}\\\hline
\m{}{}{}&\m{}{}{\bullet}&\m{}{}{}&\m{}{}{}&\m{}{}{\bullet}&\m{}{}{}&\m{}{}{\bullet}\\\hline
\m{}{}{\bullet}&\m{}{}{}&\m{}{}{}&\m{}{}{}&\m{}{}{\bullet}&\m{}{}{\bullet}&\m{}{}{}\\\hline\hline
\m{}{}{}&\m{}{}{}&\m{}{}{\bullet}&\m{}{}{\bullet}&\m{}{}{\bullet}&\m{}{}{}&\m{}{}{}\\\hline
\m{}{}{\bullet}&\m{}{}{\bullet}&\m{}{}{\bullet}&\m{}{}{}&\m{}{}{}&\m{}{}{}&\m{}{}{}\\\hline
\m{}{}{\bullet}&\m{}{}{\bullet}&\m{}{}{}&\m{}{}{\bullet}&\m{}{}{}&\m{}{}{}&\m{}{}{}\\\hline
\end{tabular}
\end{center}
We can assume that the cell $T_{15}$ contains symbols $2,3,4$. Then the first row must
contain only symbols $1,2,3,4,$ and these are distributed among the four filled cells
(in the first row) according to one of three possible ways:
\begin{center}
123, 124, 134, 234\ (or 123, 134, 124, 234)\\
124, 134, 123, 234\ (or 124, 123, 134, 234)\\
134, 124, 123, 234\ (or 134, 123, 124, 234)\\
\end{center}
The idea is to label the filled columns with one of these configurations, to label the
first row 1234, and then attempt to complete the labeling of the rows and columns
as follows:
\begin{description}
\item{$\bullet$} {each row and column is labeled by 4 elements from $\{1, \ldots , 7\}$,}
\item{$\bullet$} {the first 4 rows and first 4 columns contain 1 in its label,}
\item{$\bullet$} {first row is labeled $\{1,2,3,4\}$,}
\item{$\bullet$} {columns with filled cells in the first row are filled as above,}
\item{$\bullet$} {for any $i$, the number $i$ appears in precisely 4 row labels and in precisely 4 column labels,}
\item{$\bullet$} {if the cell $T_{ij}$ is filled, $A$ is the label of row $i$ and $B$ is the label of row $j$,
then $|A\cup B|\le5$ (because the cell $T_{ij}$ contains three elements of $A\cap B$).}
\end{description}
By applying  a depth-first search, we found no solutions (Indeed, we tried all 36 distributions of filled
cells and all three configurations in the first row). The search takes a minute with no optimization. So, it is already impossible to distribute elements
in rows and columns according to the restrictions of the $\mukm{3}{4}{7}$ disregarding how the cell
 symbols are distributed among the three components of the purported Latin trade. Therefore, there is no $\mukm{3}{4}{7}$.
}\end{proof}
\
At this point we will show the existence of some $\mukm{3}{k}{m}$s. For this purpose we will need some small cases. We have found
base rows of those Latin trades computationally, sometimes by trial and errors. But we have checked all of them by Algorithm~\ref{base-algorithm}.
\begin{theorem}\label{6,8,9,10,12}
If $k=6,8,10$ and $12$ then there exists a $\mukm{3}{k}{m}$ for every $m \geq k$.
\end{theorem}
\begin{proof}{We will show for the given $k$, there exist
 $\mukm{\mu}{k}{l}$s for $l$, where ${k+1}\leq l\leq 2k-1$. Then by
Theorem~\ref{2k-1}, we will get all $m\ge k$ where $k=6,8,10$, and
$12$.

\begin{itemize}

\item  {$k=6$.}

If $8\leq m=2l\leq 10$,
by Corollary~\ref{4-way}  a  $\mukm{3}{6}{m}$ exists.

And the following are the base rows of a $\mukm{3}{6}{m}$ for
$m=7,9,11$:

$3$--$B_7^6=\{(1,5,4)_1,(3,4,2)_2,(5,3,1)_3,(7,2,5)_4,(2,1,7)_5,(4,7,3)_6\}$,\\
$3$--$B_9^6=\{(1,8,3)_1,(3,2,1)_2,(2,5,6)_3,(6,3,2)_4,(8,6,5)_5,(5,1,8)_7\}$,\\
$3$--$B_{11}^6=\{(1,6,3)_1,(3,2,7)_2,(6,4,1)_3,(2,7,4)_4,(7,3,6)_5,(4,1,2)_{10}\}.$
\item  {$k=8$.}

If $10\leq m=2l\leq 14$,
by Corollary~\ref{4-way}  a $\mukm{3}{8}{m}$ exists.

And the following are the base rows of a $\mukm{3}{8}{m}$ for
$m=9,11,13,15$:

$3$--$B_9^8=\{(1,8,7)_1,(3,2,9)_2,(2,4,3)_3,(7,1,6)_4,(9,7,4)_5,(8,9,1)_6,(4,6,8)_7,$\\
\hspace*{1.62cm}$(6,3,2)_8\}$,

$3$--$B_{11}^8=\{(1,5,4)_1,(3,2,11)_2,(2,4,5)_3,(6,1,3)_4,(8,3,2)_5,(4,8,6)_6,(11,6,8)_7,$\\
\hspace*{1.62cm}$(5,11,1)_8\}$,

$3$--$B_{13}^8=\{(1,5,3)_1,(3,1,5)_2,(2,6,11)_3,(6,4,2)_4,(8,3,4)_5,(4,8,6)_6,(11,2,8)_7,$\\
\hspace*{1.62cm}$(5,11,1)_{10}\}$,

$3$--$B_{15}^8=\{(1,11,4)_1,(3,2,6)_2,(2,4,3)_3,(6,7,2)_4,(8,3,7)_5,(4,8,1)_6,(11,6,8)_7,$\\
\hspace*{1.62cm}$(7,1,11)_{12}\}.$

\item  {$k=10$.}

If $12\leq m=2l\leq 18$,
 by Corollary~\ref{4-way}  a  $\mukm{3}{10}{m}$ exists.

And the following are the base rows of a $\mukm{3}{10}{m}$ for
$m=13,15,17, 19$:

$3$--$B_{13}^{10}=\{(1,11,6)_1,(3,2,13)_2,(2,4,3)_3,(6,8,7)_4,(8,7,4)_5,(4,5,2)_6,(11,3,8)_7,$\\
\hspace*{1.62cm}$(13,6,5)_8,(5,1,11)_9,(7,13,1)_{10}\}$,

$3$--$B_{15}^{10}=\{(1,6,5)_1,(3,2,4)_2,(2,4,14)_3,(6,8,3)_4,(8,1,2)_5,(4,3,6)_6,(11,5,8)_7,$\\
\hspace*{1.62cm}$(5,7,11)_8,(14,11,7)_9,(7,14,1)_{11}\}$,


$3$--$B_{17}^{10}=\{(1,6,4)_1,(3,2,6)_2,(2,7,14)_3,(6,1,2)_4,(8,4,5)_5,(4,8,3)_6,(11,5,8)_7,$\\
\hspace*{1.62cm}$(5,11,7)_8,(14,3,11)_9,(7,14,1)_{13}\}$,

$3$--$B_{19}^{10}=\{(1,6,2)_1,(3,2,6)_2,(2,4,14)_3,(6,8,7)_4,(8,7,3)_5,(4,3,5)_6,(11,5,4)_7,$\\
\hspace*{1.62cm}$(5,11,8)_8,(14,1,11)_9,(7,14,1)_{15}\}.$

\item  {$k=12$.}

If $14\leq m=2l\leq 22$ or $m=15,21$,
 by Corollary~\ref{4-way}  a
$\mukm{3}{12}{m}$ exists.



And the following are the base rows of a $\mukm{3}{12}{m}$ for $
m=17,19,23$:

$3$--$B_{17}^{12}=\{(1,16,4)_1,(3,7,2)_2,(2,4,14)_3,(6,8,3)_4,(8,5,11)_5,(4,3,10)_6,(11,1,8)_7$, \\
\hspace*{1.62cm}$(5,14,6)_8,(14,11,5)_9,(16,6,7)_{10},(7,10,16)_{11},(10,2,1)_{16}\}$,

$3$--$B_{19}^{12}=\{(1,16,7)_1,(3,2,6)_2,(2,4,3)_3,(6,9,1)_4,(8,7,4)_5,(4,3,11)_6,(11,5,2)_7,$ \\
\hspace*{1.62cm}$(5,11,9)_8,(14,8,5)_9,(16,14,8)_{10},(7,6,14)_{11},(9,1,16)_{14}\}$,

$3$--$B_{23}^{12}=\{(1,7,5)_1,(3,2,8)_2,(2,4,1)_3,(6,9,3)_4,(8,1,7)_5,(4,3,9)_6,(11,5,2)_7,$ \\
\hspace*{1.62cm}$(5,11,4)_8,(14,8,6)_9,(16,14,11)_{10},(7,6,16)_{11},(9,16,14)_{14}\}.$
\end{itemize}
\vspace*{-11mm}}\end{proof}

\subsection{Small  odd $k$}

\begin{proposition}\label{5}
There exists a $\mukm{3}{5}{m}$ for every $m \geq 5$, except
possibly $m=6$.
\end{proposition}
\begin{proof}{
By Lemma~\ref{k,k} there exists a  $\mukm{3}{5}{5}$.
The following are the base rows of a $\mukm{3}{5}{m}$ for $m=7,8,9,11$:

\noindent $3$--$B_7^5=\{(1,3,2)_1,(3,2,5)_2,(5,7,3)_3,(7,5,1)_4,(2,1,7)_5\}$,\\
$3$--$B_8^5=\{(1,6,2)_1,(3,2,4)_2,(2,4,3)_3,(6,3,1)_4,(4,1,6)_7\}$,\\
$3$--$B_9^5=\{(1,4,3)_1,(4,3,8)_2,(7,1,4)_4,(3,8,7)_6,(8,7,1)_7\}$,\\
$3$--$B_{11}^5=\{(1,6,9)_1,(9,2,11)_5,(11,1,6)_6,(2,11,1)_7,(6,9,2)_9\}.$

\noindent By Theorem~\ref{k(rm+sn)}, a $\mukm{3}{5}{10}$ exists. So a  $\mukm{3}{5}{m}$         
exists for 5 consecutive values $m\in \{7,8,\ldots,11\}$. Thus a  $\mukm{3}{5}{m}$              
exists for all $m\ge 7$ by Theorem~\ref{k(rm+sn)}.                                                                        
%
}\end{proof}
\begin{theorem}\label{7,9,11,13}
If $k=7,9,11$ and $13 $ then there exists a $\mukm{3}{k}{m}$ for every $m \geq k$.
\end{theorem}
\begin{proof}{We introduce the following base rows:

\begin{itemize}

\item  {$k=7$.}
\begin{itemize}
\item [$m\geq 8$:]
$3$--$B_m^7=\{(1,4,2)_1,(3,1,4)_2,(2,3,6)_3,(6,5,1)_4,(8,2,3)_5,(4,8,5)_6,(5,6,8)_8\}.$
\end{itemize}
%
%
\item  {$k=9$.}
\begin{itemize}
\item [$m=10$:]
$3$--$B_{10}^9=\{(1,7,9)_1,(3,2,8)_2,(2,4,5)_3,(7,6,4)_4,(9,3,2)_5,(8,9,7)_6,(4,1,6)_7,$\\
\hspace*{1.62cm}$(6,5,1)_8,(5,8,3)_9\}.$

\item [$m\geq 11$:]
$3$--$B_m^9=\{(1,5,4)_1,(3,4,6)_2,(2,3,1)_3,(6,2,5)_4,(8,1,2)_5,(4,7,8)_6,(11,6,3)_7,$\\
\hspace*{1.62cm}$(5,11,7)_8,(7,8,11)_{11}\}.$

 \end{itemize}
\item  {$k=11$.}
\begin{itemize}
\item [$m\geq 11$:]
$3$--$B_m^{11}=\{(6,1,2)_1,(1,7,4)_2,(7,2,1)_3,(2,8,7)_4,(8,3,10)_5,(3,9,5)_6,(9,4,11)_7,$\\
\hspace*{1.62cm}$(4,10,3)_8,(10,5,9)_9,(5,11,6)_{10},(11,6,8)_{11}\}.$
\end{itemize}

\item  {$k=13$.}
\begin{itemize}
\item [ $m\geq 13$:]
$3$--$B_m^{13}=\{(7,1,2)_1,(1,8,4)_2,(8,2,1)_3,(2,9,3)_4,(9,3,8)_5,(3,10,11)_6,(10,4,13)_7,$\\
\hspace*{1.62cm}$(4,11,12)_8,(11,5,6)_9,(5,12,10)_{10},(12,6,5)_{11},(6,13,7)_{12},(13,7,9)_{13}\}.$
\end{itemize}
 \end{itemize}
\vspace*{-7mm}}\end{proof}
\begin{theorem}\label{15}
If $k=15$ and $m\geq 15 $ then there exists a $\mukm{3}{15}{m}$.
\end{theorem}
\begin{proof}{
By Lemma~\ref{k,k}, Theorem~\ref{k(k+1)} and  Corollary~\ref{4-way}, we have a
$\mukm{3}{15}{m}$ for $m=15,16,18$ and $20$. The following is a base
row of a $\mukm{3}{15}{m}$ for $m\ge 21$:

$3$--$B_{m}^{15}=\{(1,5,4)_1,(3,1,2)_2,(2,9,11)_3,(6,11,3)_4,(8,6,7)_5,(4,14,10)_6,(11,4,8)_7,$ \\
\hspace*{2.15cm}$(5,3,6)_8,(14,7,5)_9,(16,10,1)_{10},(7,2,16)_{11},(19,8,9)_{12},(21,16,19)_{13},$ \\
\hspace*{2.15cm}$(9,19,21)_{14},(10,21,14)_{19}\}.$

The following are the base rows of a $\mukm{3}{15}{m}$ for
$m=17,19$:

$3$--$B_{17}^{15}=\{(5,2,12)_3,(7,15,11)_4,(9,17,4)_5,(11,13,14)_6,(13,16,5)_7,(15,11,13)_8,$ \\
\hspace*{2.15cm}$(17,14,6)_9,(2,12,16)_{10},(4,9,7)_{11},(6,8,15)_{12},(8,10,17)_{13},$ \\
\hspace*{2.15cm}$(10,7,8)_{14},(12,6,10)_{15},(14,5,9)_{16},(16,4,2)_{17}\},$

$3$--$B_{19}^{15}=\{(1,2,11)_1,(3,4,2)_2,(5,17,4)_3,(7,10,9)_4,(9,15,14)_5,(11,9,13)_6,$ \\
\hspace*{2.15cm}$(13,19,10)_7,(15,13,5)_8,(17,6,1)_9,(19,14,3)_{10},(2,11,15)_{11},$ \\
\hspace*{2.15cm}$(4,1,7)_{12},(6,3,19)_{13},(10,7,17)_{15},(14,5,6)_{17}\}.$}\end{proof}

\subsection{General cases}

\begin{theorem}\label{m-2}
Let $m \equiv 1 ({\rm mod \ 6})$ and $m\ge 7$. Then there exists a
$(3,m-2,m)$ Latin trade.
\end{theorem}
\begin{proof}{The following is a base row of a  $\mukm{2}{m-2}{m}$:

$2$--$B_m^{m-2}= \bigcup_{i=0}^{(m-13)/6}\{(6i+2,6i+3)_{3i+1},(6i+4,6i+2)_{3i+2},(6i+3,6i+4)_{3i+3},$\\
\hspace*{2.15cm}$(6i+5,6i+6)_{(m+3)/2+3i+1},(6i+7,6i+5)_{(m+3)/2+3i+2},(6i+6,6i+7)_{(m+3)/2+3i+3}\}$\\
\hspace*{2.15cm}$\bigcup
\{(m-5,m-4)_{(m-7)/2+1},(m-2,m-5)_{(m-7)/2+2},(m-4,m)_{(m-7)/2+3},$\\
\hspace*{2.15cm}$(1,m-2)_{(m-7)/2+4},(m,1)_{(m-7)/2+5}\}.$

\noindent Now, for $1\le i\le m-2$ we put $2i-1 ({\rm mod \ m})$
in $i$-th cell of $2$--$B_m^{m-2}$, as a result we obtain a  base
row of a $\mukm{3}{m-2}{m}$. }\end{proof}

\begin{example}
As an example of the previous theorem, the following is a base row
of a $\mukm{3}{11}{13}$:

$3$--$B_{13}^{11}=\{(1,2,3)_1,(3,4,2)_2,(5,3,4)_3,(7,8,9)_4,(9,11,8)_5,(11,9,13)_6,(13,1,11)_7,$\\
\hspace*{2.15cm}$(2,13,1)_8,(4,5,6)_9,(6,7,5)_{10},(8,6,7)_{11}\}.$

\end{example}

\begin{theorem}\label{5m}
For every $m=5l$ and $4\leq k \leq m$, $l\neq 6$,  there exists
a $(3,k,m)$ Latin trade.
\end{theorem}
\begin{proof}{
The theorem trivially holds for $l=1$. If $l=2$, then by
Theorem~\ref{k(k+1)}, Theorem~\ref{k(rm+sn)},
Theorem~\ref{7,9,11,13} and Theorem~\ref{6,8,9,10,12}, we can
construct a $\mukm{3}{k}{10}$ for every $4\leq k\leq 10$. By Theorems~\ref{6,8,9,10,12} and~\ref{7,9,11,13} there
exists a $\mukm{3}{k}{m}$ for $k=6,7$ and $11$, so suppose that
$k\neq 6,7$ and $11$.

\noindent  We may also assume that
$m>k$.\\
 We have the following cases to consider, each case follows from
 Theorem~\ref{sum}:
\begin{itemize}
\item  {$k=5l^{'}$.} \\
We set $ k_i=5$ for $1 \leq i \leq l^{'}$ and $k_i=0$ for
$l^{'}+1 \leq i \leq l$ and $p=5$.

\item  {$k=5l^{'}+1$.} \\
We set $ k_i=5$ for $1 \leq i \leq l^{'}-3 $ and $k_i=4$ for
$l^{'}-2\leq i \leq l^{'}+1$ and $k_i=0$ for $l^{'}+2 \leq i \leq
l$ and $p=5$.

\item  {$k=5l^{'}+2$.} \\
We set $ k_i=5$ for $1 \leq i \leq l^{'}-2 $ and $k_i=4$ for
$l^{'}-1\leq i \leq l^{'}+1$
 and $k_i=0$ for $l^{'}+2 \leq i \leq l$ and $p=5$.

\item  {$k=5l^{'}+3$.} \\
We set $ k_i=5$ for $1 \leq i \leq l^{'}-1$,
$k_{l^{'}}=k_{l^{'}+1}=4$, and $k_i=0$ for $l^{'}+2 \leq i \leq
l$, and $p=5$.

\item  {$k=5l^{'}+4$.} \\
We set $ k_i=5$ for $1 \leq i \leq l^{'}$, $k_{l^{'}+1}=4$, and
$k_i=0$ for $l^{'}+2 \leq i \leq l$, and $p=5$.
\end{itemize}
\vspace*{-0.7cm}  }\end{proof}
%
%
\begin{theorem}\label{7m}
For every $m=7l$ and $5\leq k \leq m$, $l\neq 6$,   there exists
a $(3,k,m)$ Latin trade.
\end{theorem}
\begin{proof}{The theorem  trivially holds for $l=1$.
If $l=2$, then by Theorem~\ref{k(rm+sn)}, Theorem~\ref{7,9,11,13}
and Theorem~\ref{6,8,9,10,12}, we can construct a
$\mukm{3}{k}{14}$ for every $5\leq k\leq 14$. For $l\neq 2,6$ by
Theorems~\ref{6,8,9,10,12} and~\ref{7,9,11,13} there exists a
$\mukm{3}{k}{m}$ for $k=8,9$, so suppose that $k\neq 8,9$.

\noindent  We may also assume that
$m>k$.\\
 We have the following cases to consider, each case follows from
 Theorem~\ref{sum}:
\begin{itemize}
\item  {$k=7l^{'}$}.\\
We set $ k_i=7$ for $1 \leq i \leq l^{'}$ and $k_i=0$ for $l^{'}+1
\leq i \leq l$ and $p=7$.

\item  {$k=7l^{'}+1$}.\\
We set $k_i=7$ for $1 \leq i \leq l^{'}-2$ and $k_i=5$ for
$l^{'}-1\leq i \leq l^{'}+1$ and $k_i=0$ for $l^{'}+2 \leq i \leq
l$ and $p=7$.

\item  {$k=7l^{'}+2$}.\\
We set $k_i=7$ for $1 \leq i \leq l^{'}-2 $ and
$k_{l^{'}-1}=k_{l^{'}}=5$,  $k_{l^{'}+1}=6$
 and $k_i=0$ for $l^{'}+2 \leq i \leq l$ and $p=7$.

\item  {$k=7l^{'}+3$}.\\
We set $k_i=7$ for $1 \leq i \leq l^{'}-1$,
$k_{l^{'}}=k_{l^{'}+1}=5$, and $k_i=0$ for $l^{'}+2 \leq i \leq
l$, and $p=7$.

\item  {$k=7l^{'}+4$}.\\
We set $ k_i=7$ for $1 \leq i \leq l^{'}-1$, $k_{l^{'}}=6$,
$k_{l^{'}+1}=5$ and $k_i=0$ for $l^{'}+2 \leq i \leq l$, and
$p=7$.

\item  {$k=7l^{'}+5$}.\\
We set $ k_i=7$ for $1 \leq i \leq l^{'}$, $k_{l^{'}+1}=5$, and
$k_i=0$ for $l^{'}+2 \leq i \leq l$, and $p=7$.

\item  {$k=7l^{'}+6$}.\\
We set $ k_i=7$ for $1 \leq i \leq l^{'}$, $k_{l^{'}+1}=6$, and
$k_i=0$ for $l^{'}+2 \leq i \leq l$, and $p=7.$
\end{itemize}
\vspace*{-5mm}}\end{proof}

Now by the results given above we have proved Theorem~\ref{main}, given at the end of the Introduction.
\\

%
%
\noindent {\bf Acknowledgement.} We  appreciate very useful comments
of anonymous referee, who also added the proof of Proposition~\ref{(3,4,7)}.
 Also we thank to Amir Hooshang Hosseinpoor and Mayssam Mohammadi Nevisi
 for their computer programming.


\begin{thebibliography}{10}

\bibitem{BagheriMah}
Behrooz Bagheri~Gh. and E.~S. Mahmoodian.
\newblock On the existence of $k$--homogeneous {L}atin bitrades.
\newblock {\em Util. Math.}, 85:333--345, 2011.

\bibitem{MR2220235}
Richard Bean, Hoda Bidkhori, Maryam Khosravi, and E.~S. Mahmoodian.
\newblock {$k$}-homogeneous {L}atin trades.
\newblock {\em Bayreuth. Math. Schr.}, 74:7--18, 2005.

\bibitem{MR2041871}
Elizabeth~J. Billington.
\newblock Combinatorial trades: a survey of recent results.
\newblock In {\em Designs, 2002}, volume 563 of {\em Math. Appl.}, pages
  47--67. Kluwer Acad. Publ., Boston, MA, 2003.
\bibitem{MR2170114}
Nicholas Cavenagh, Diane Donovan, and Ale{\v{s}} Dr{\'a}pal.
\newblock 3-homogeneous {L}atin trades.
\newblock {\em Discrete Math.}, 300(1-3):57--70, 2005.

\bibitem{MR2139816}
Nicholas Cavenagh, Diane Donovan, and Ale{\v{s}} Dr{\'a}pal.
\newblock 4-homogeneous {L}atin trades.
\newblock {\em Australas. J. Combin.}, 32:285--303, 2005.

\bibitem{MR2453264}
Nicholas~J. Cavenagh.
\newblock The theory and application of {L}atin bitrades: a survey.
\newblock {\em Math. Slovaca}, 58(6):691--718, 2008.

\bibitem{MR2563279}
Nicholas~J. Cavenagh and Ian~M. Wanless.
\newblock On the number of transversals in {C}ayley tables of cyclic groups.
\newblock {\em Discrete Appl. Math.}, 158(2):136--146, 2010.

\bibitem{MR2246267}
Charles~J. Colbourn and Jeffrey~H. Dinitz, editors.
\newblock {\em Handbook of combinatorial designs}.
\newblock Discrete Mathematics and its Applications (Boca Raton). Chapman \&
  Hall/CRC, Boca Raton, FL, second edition, 2007.

\bibitem{MR2048415}
A.~D. Keedwell.
\newblock Critical sets in {L}atin squares and related matters: an update.
\newblock {\em Util. Math.}, 65:97--131, 2004.

\bibitem{Lindner}
C.~C. Lindner and C.~A. Rodger.
\newblock {\em Design theory}.
\newblock CRC Press LLC, 1997.

\end{thebibliography}
\def\cprime{$'$}

 \end{document}